\documentclass[12pt,a4]{article}

\usepackage{amsmath,graphicx,amsthm,amssymb,mathrsfs} 

\newcommand{\exn}{{\bf E}}
\newcommand{\pr}{{\bf P}}
\newcommand{\N}{\mathbb{N}}
\newcommand{\R}{\mathbb{R}}

\newcommand{\Ic}{\mathcal{I}}

\newcommand{\Ac}{\mathcal{A}}

\newcommand{\al}{\alpha}

\newcommand{\bs}[1]{\boldsymbol{#1}}

\newcommand{\Be}{{\rm B}}

\newcommand{\de}{\delta}

\newcommand{\deq}{\stackrel{d}{=}}

\newcommand{\Va}{\mbox{\rm Var}\, }
\newcommand{\Co}{\mbox{\rm Cov}\, }
\newcommand{\ind}{{\bf 1}}

\newcommand{\toas}{\stackrel{\footnotesize\mbox{a.s.}}{\longrightarrow}}

\newtheorem{rema}{Remark}
\newtheorem{theo}{Theorem}

\newtheorem{lemo}{Lemma}
\newtheorem{coro}{Corollary}

\title{Gaussian process approximations for multicolor P\'olya urn models}
\author{Konstantin Borovkov$^1$ }

\begin{document}
	\maketitle

\footnotetext[1]{School of Mathematics and Statistics, The University of Melbourne, Parkville 3010, Australia; e-mail: borovkov@unimelb.edu.au.}

\begin{abstract}
Motivated by mathematical tissue growth modelling, we consider the problem of approximating the dynamics of multicolor P\'olya urn processes that start with large numbers of balls of different colors and run for a long time. Using strong approximation theorems for empirical and quantile processes, we establish Gaussian process approximations for the P\'olya urn processes. The approximating processes are sums of a multivariate Brownian motion process and  an independent linear drift with a random Gaussian coefficient. Which of the two terms dominates depends on the ratio of the number of time steps~$n$ to the initial   number  of balls~$N$ in the urn. We also establish an upper bound of the form $c(n^{-1/2}+N^{-1/2})$ for the maximum deviation over the class of convex Borel sets of the step~$n$ urn composition distribution from the approximating normal  law.

		
	\end{abstract}

\section{Introduction and main results}

The main motivation for this work came from the  author's conversations with K.A.~Landman concerning certain mathematical tissue growth models she was  developing with her co-authors some  time ago~\cite{BiLaSi08, BiLa09}. One of the purposes of the  models was to better understand the neural crest invasion process, which occurs during embryogenesis and leads to the  formation of the enteric nervous system in the intestine of vertebrates. A major contributing factor to that process is gut growth; of interest is understanding the position and speed of the wavefront with time and tracking individual cells or small groups of cells at different positions in the wave. Idealizing the developing three-dimensional gut tissue as a series of cylindrical shells, the authors of~\cite{BiLaSi08} assumed further that ``the thickening of the
cross-sectional area is small compared to the elongation and
radial expansion, so that the overall growth can be approximated
by the increase in the area of the outer cylindrical
surface. Each cylindrical shell is cut lengthwise, so becoming
a flat rectangular plate with a fixed thickness." The basic model only considers uniaxial growth, the width of the rectangular remaining constant.

The key element of this basic approach   discussed in~\cite{BiLaSi08, BiLa09} is the following continuous time one-dimensional cellular agent model. Suppose we start with $L(0)=N$ agents located at the integer points $1,2,\ldots, L(0)$ on the real line. Some of the agents are marked, and one is interested in tracking the movement  of these marked agents as the tissue grows. The agents are allowed to proliferate, by
mitotic division.   At the end of an agent's life, it splits into two (modeling mitotic division of cells), the daughter agents being inserted in the ``linear tissue" instead of the mother. If the latter was at location $i$ at its division time~$t>0$, then its daughters will be placed  at locations $i$ and $i+1$, while all the agents that prior to that division event were at locations $i+1,i+2, \ldots, L(t-)$ will be  pushed one step to the right, to the respective locations $i+2,i+3, \ldots, L(t)=L(t-)+1$. For the definiteness' sake, if the mother agent was marked, the mark is inherited   by the ``right daughter" (the one at location $i+1$) only.

The lifetimes of the agents are assumed to be independent exponential random variables with a common rate~$\lambda >0$, so that the dynamics of $L(t)$ is that of a pure birth process with the birth rate $\lambda k$ given $L(t)=k.$ This is the well-known Yule process. The dynamics of the process are well-understood: basically, it is exponential growth, see, e.g.,~\cite{de06} and references therein.

Assume that  $d$ of the initial $N$ tissue cells were marked for tracking. Denote  the locations of the marked agents immediately after the $n$th division event  by  $ M_1(n) <M_2 (n) <\cdots <M_d (n) ,$   their initial values being $M_j(0)=N_j,$ $j=1,\ldots , d,$ where  $0=:N_0 <N_1<N_2 <\cdots <N_d <N=:N_{d+1}$, and set for notational convenience $M_0(n):=0,$ $M_{d+1}(n):= N+n.$ As the rates of all the agents' exponential lifetimes are the same, all the currently existing agents are equally likely to become the next cell to divide. Interpreting the
\[
\xi_j(n):=M_j(n)-M_{j-1}(n), \quad j=1,\ldots, d+1,
\]
agents located at the integer points within $(M_{j-1}(n),M_j(n)]$ as ``balls of color $j$" after $n$ steps in an urn model, we see that the dynamics of
\begin{align}
\label{xi}
\{\bs{\xi}(n):= (\xi_1(n),\ldots, \xi_{d+1} (n)): n\ge 0\}
\end{align}
are those     of a simple $(d+1)$-color P\'olya urn, where at each step a ball is extracted at random from the urn  and then returned back with one more ball of the same color. This reduction was noted in~\cite{BiLaSi08} in the case of~$d=1.$ In fact, the reverse ``embedding" of discrete time urn schemes into Markovian continuous time models was used already about fifty years ago in~\cite{AtKa68} (see also Section~4 in~\cite{AtNe72} and a general discussion in~\cite{BlHo91}) and more recently was   successfully revisited and further developed in~\cite{Ja04,Ja06}.

To analyze the cellular agents model, the authors of~\cite{BiLa09} ran simulations, starting with $N=24$ initial agents. The results  showed emergence of bell-shaped distribution curves (quite Gaussian in their shape) for the marked agents' locations after a few tens of divisions.  Then, using the conditional (given the current urn composition) first and second moments of the one-step increments of the number of balls of a given color  in a two-color (i.e., when $d=1$) continuous time P\'olya urn model, the authors  wrote down the Fokker--Planck equation with the same infinitesimal drift and diffusion coefficients. The plots of the densities of the corresponding  diffusion process proved to be in excellent agreement with the simulation results, suggesting that the heuristic approach used   indicated a valid mathematical approximation result.

In the present paper, we establish Gaussian process  approximations to the trajectories  of the locations of  $d\ge 1$ marked agents in the uniaxial growth  cellular agent model, under the most relevant to the above motivation  scenario where the initial number of cells~$N$ is large  (\cite{BiLa09} mentions that in quail, the cell
number in the small plus large intestine is around
800,000 at embryonic age four days and increases   approx.\ five-fold in the next three days). We  measure time in terms of the number of cell divisions  $n\to\infty$  rather than using  the ``physical time"~$t$ (to which transition is quite straightforward). In other words, we   provide approximations to the dynamics of a $(d+1)$-color P\'olya urn model, as both the initial number of balls~$N$ and the number~$n$ of draws tend to infinity. These results are based in the strong approximation established in Theorem~\ref{T1} and are stated in Corollary~\ref{C1} below.
Furthermore, in Theorem~\ref{T2} we establish a uniform  bound for the rate of convergence of the distribution of the time $n$ locations of the~$d$ marked agents (or, equivalently, for  the urn composition after $n$ draws) to the respective normal distribution on~$\R^d.$

There exists vast literature devoted to studying the behavior of what is nowadays called ``P\'olya urn models" and their generalizations. In its basic form, the model was introduced in 1906 in Section~4 of the very first paper on chain dependence by A.A.~Markov~\cite{Ma06}. It was used there  as an example of a sequence of dependent random variables for which the weak law of large numbers did not hold. The standard   reference  being the 1923 paper by F.~Eggenberger and G.~P\'olya~\cite{EgPo23}, the reader is further  referred to~\cite{JoKo77, KoBa97,Ma09} for surveys of results concerning  urn models. We already mentioned papers~\cite{Ja04, Ja06} which present several limit theorems for the numbers of balls in  generalized two-color  P\'olya urn schemes and provide   references to earlier work in that direction, the former paper using  Athreya--Karlin's embedding idea from~\cite{AtKa68}. However, to the best of our knowledge, no  limit theorems were available in the important for us case of large numbers of initial balls.

One of the key classical results for the basic two-color  P\'olya urns  is the a.s.\ convergence (as the number of draws $n$ goes to infinity) of the proportion of the balls of a given color to a beta-distributed random variable (and we note that it was shown in~\cite{GoRe13} that  the rate of the distributional  convergence in Wasserstein metric is $O(n^{-1}),$  with a constant depending on the initial composition of the urn).  A similar convergence result holds for  multi-color urn models as well. For $\bs{\al} := (\al_1 ,\ldots, \al_{d+1}  )\in \N^{d+1}_{>0},$ denote by $\bs{\xi}^{\bs \al} (n)$    the value of the   $(d+1)$-color P\'olya urn process~\eqref{xi} after~$n$ draws, given that it  started  at the initial vector $\bs{\xi}^{\bs \al} (0) =\bs{\al}$ (note that the dimensionality of the superscript gives the number of colors in the model and so completely specifies the latter; for brevity, the superscript $\bs{\al}$ will mostly be omitted in what follows).  Then, for the vector of the proportions of balls of different colors after~$n$ draws, one has
\begin{equation}
\bs{\xi}^{\bs{\al}}  (n)/(N+n)\toas \bs{V}^{\bs{\al}} \quad \mbox{as}\quad n\to\infty,
\label{Conv_to_dir}
\end{equation}
where    $N =\sum_{j=1}^{d+1} \al_j$  and the limiting random vector  $\bs{V}^{\bs{\al}} $ (for which we will also often omit the superscript~$\bs{\al}$ writing just $\bs{V}=(V_1,\ldots,V_{d+1})$) follows the Dirichlet distribution $\mbox{Dir}_{d+1}(\bs{\al})$ with parameters $(\al_1, \ldots,\al_{d+1})$ (see~\cite{BlMa73}; note that the a.s.\ convergence here  is a chrestomathic consequence of the martingale convergence theorem). Introducing notations
\begin{align*}
\bs{x}^\#:= (x_1, \ldots, x_{d})\in \R^{d}
\quad \mbox{for}\quad \bs{x}=(x_1, \ldots, x_d, x_{d+1} )\in \R^{d+1},
\end{align*}
and $B^\#:= \{\bs{x}^\#: \bs{x}\in B\}$ for $B\in \R^{d+1},$ recall that  $\mbox{Dir}_{d+1}(\bs{\al})$ is concentrated on the standard simplex
\[
\Delta^d:=\Bigl\{\bs{x}=(x_1, \ldots, x_{d+1})\in \R^{d+1}_{>0}: \sum_{j=1}^{d+1}x_j=1 \Bigr\}
\]
and is such that the truncated vector $\bs{V}^\#$  has density
\[
\frac1{{\rm B}(\bs{\al})}\prod_{j=1}^{d+1}x_j^{\al_j-1},\quad x_{d+1 }:= 1- \sum_{j=1}^{d}x_j,
\]
on the ``corner''
 \[
\Delta^{d\#} =\Bigl\{\bs{x}=(x_1, \ldots, x_{d})\in \R^{d}_{>0}: \sum_{j=1}^{d}x_j < 1\Bigr\}.
\]
Here
\[
{\rm B}(\bs{z}) :=\frac{\prod_{j=1}^{d+1} \Gamma (z_j)}{ \Gamma (\sum_{j=1}^{d+1}z_j)}, \quad \bs{z}=(z_1,\ldots,z_{d+1})\in\R^{d+1}_{>0},
\]
is the multivariate beta function.

Keeping in mind the motivation for this work (and also for simplicity's sake), we will state our   main results in a form admitting direct interpretation in terms of the dynamics of the markers in the growing tissue. Introduce the following notation for vectors of partial sums: for $\bs{x}=(x_1,\ldots,x_k)\in \R^k$, we set
\begin{align*}
\widehat{\bs{x}}:=(\widehat x_1,\ldots,\widehat x_k), \quad \mbox{where}\quad \widehat x_j=\sum_{i=1}^j x_j, \quad j=1,\ldots, k,
\end{align*}
and let $\widehat{x}_0:=0.$ For $k,m\in \N_{>0}$ such that $2\le k\le m,$ set
\[
\Lambda_m^k := \{\bs{x}\in \N^k_{>0}: \widehat{x}_k=m\},
\quad
\Lambda_m:=\bigcup_{2\le k \le m}\Lambda_m^k,
\quad
\Lambda :=\bigcup_{m \ge 2}\Lambda_m.
\]

Now recall that, in terms of the urn process~\eqref{xi}, the locations of the~$d$ markers after $n\ge 0$ cell divisions are
\begin{align}
\label{N_xi}
M_j (n)= \widehat{\xi}_j (n),\qquad j=1,\ldots , d,
\end{align}
with   $M_j (0)= N_j = \widehat{\al}_j$ being the initial locations of the markers given that the initial numbers of balls of different colors in the  urn are  specified  by the vector $\bs{\al} \in \Lambda^{d+1}_N.$  We set
\[
\bs{M}(n):= (M_1 (n), \ldots, M_{d } (n))\in\N^{d }_{>0}
\]
(recalling that we let  $M_{d+1}(n):=\widehat \xi_{d+1}(n)= N+n,$ $n\ge 0$) and $\bs{N} := (N_1  , \ldots, N_{d }  ).$

Our first assertion is based on the Blackwell--MacQueen theorem~\cite{BlMa73} (closely related to the Hewitt--Savage theorem on exchangeable random variables~\cite{HeSa55}) and the  strong approximation results for empirical and quantile processes. We  will need some further notations. By $\{W^0(y): y\in [0,1]\}$ we will denote the standard Brownian bridge process, i.e., a continuous zero mean Gaussian process with covariance function $\exn  W^0(y_1)W^0(y_2)=y_1\wedge y_2 -y_1 y_2, $ $y_1, y_2\in [0,1],$  and by $\{K(y,t): y\in [0,1], t\ge 0\}$ the ($2$-parameter) Kiefer process, which is a continuous zero mean Gaussian field with covariance function
\begin{align}
\label{Kiefer_cov}
\exn K(y_1,t_1) K(y_1,t_1) = (y_1 \wedge y_2 - y_1 y_2)(t_1\wedge t_2),
\quad y_1, y_2\in [0,1], \quad t_1, t_2\ge 0.
\end{align}
To help one ``visualize" the Kiefer process, note that, for any fixed $y\in [0,1]$, one has the equality in distribution $\{K(y,t):   t\ge 0 \}\deq \{( y(1-y))^{1/2}W(t): t\ge 0\}, $
where $W$ is the standard Wiener process, whereas for any fixed $t>0$ one has
\begin{align*}
\{K(y,t):   y\in [0,1]  \}\deq \{t^{1/2}W^0(y):  y\in [0,1] \}.
\end{align*}

\begin{theo}
	\label{T1}
One can construct all the processes from the family  $\{\bs{\xi}^{\bs\al}: \bs{\al}\in \Lambda\} $  on a common probability space together with a sequence of Brownian bridges $\{ W^{0,N}: N\ge 1\}$ and an independent of that sequence Kiefer process $K$ such that, for   $N>d\ge 1$ and $n\ge 2,$  for the partial sums~\eqref{N_xi} of the components of the processes $\bs{\xi}=\bs{\xi}^{\bs{\al}},$ $\bs{\al}\in \Lambda_N^{d+1},$ one has
\begin{multline}
	\label{Main}
M_j (n) =
  ( N + n)\widehat{\mu}_j 
+\frac{n}{N^{1/2}}  W^{0,N} (\widehat{\mu}_j  ) +\frac{n \ln N}{N }  R_j
\\
  + K\Bigl(\widehat{\mu}_j
  + \frac{1}{N^{1/2}}  W^{0,N} ( \widehat{\mu}_j )
  +\frac{  \ln N}{N } R_j , n\Bigr)+   R_j^* \ln^2 n  ,\quad j=1,\ldots, d,
\end{multline}
where
\[
\bs{\mu} = \bs{\mu}^{ \bs{\al}}:= \bs{\al} / \widehat{\al}_{d+1}=\bs{\al} /N \in\Delta^d
\]
and the remainder terms $R_j=R_j (\bs{\al} )$ and $R_j^*=R_j^* (\bs{\al},n)$ satisfy
\begin{align}
	\label{R1_error}
	\limsup_{N\to\infty}
	\max_{1\le d < N}
	\max_{\bs{\al}\in \Lambda^{d+1}_N}
	\max_{1\le j\le d} |R_j |
\le C  ,
\\
\label{R1*_error}
 \limsup_{n\to \infty}\sup_{N\ge 2}
 \max_{1\le d < N}
 \max_{\bs{\al}\in {\Lambda}^{d+1}_N}
 \max_{1\le j\le d}  |R_j^* | \le C
	\end{align}
a.s.\ for some absolute constant $C<\infty$.
\end{theo}

The first  term on the right-hand side of~\eqref{Main} represents  the mean growth  of the tissue along its longitudinal axis (which is linear  in the ``cell-division counter time"~$n$, but exponential in    ``real time"). The random part of the approximation on the right-hand side of~\eqref{Main} has different dominating terms depending on the relationship between the initial tissue  length~$N$ and the number~$n$ of cell divisions. Corollary~\ref{C1} below  presents our   findings concerning the functional limit theorems for the urn processes and covers the whole spectrum of possible limiting behaviors.
To state it, for $\bs{\al}\in \Lambda_N^{d+1}$ we introduce the $(d\times d)$-matrix
\[
\Sigma^{\bs{\al}}:=
 \bigl(\widehat{\mu}_i  \wedge \widehat{\mu}_j - \widehat{\mu}_i\widehat{\mu}_j: 1\le i,j\le d\bigr),
\quad\mbox{where}\  \bs{\mu} = \bs{\mu}^{ \bs{\al}} .
\]

\begin{coro}
	\label{C1}
One can construct  all the processes from the  family $\{\bs{\xi}^{\bs\al}: \bs{\al}\in \Lambda\} $  on a common probability space together with a family of multivariate  Brownian motion processes $\{\bs{W}^{\bs \al}  : \bs{\al}\in \Lambda \} $ with zero drift and respective covariance matrices $\Sigma^{\bs{\al}}$
and an independent of these Brownian motion processes family of Gaussian random vectors $\{ \bs{Z}^{\bs{\al}} : \bs{\al}\in \Lambda \} $ with zero mean and the same respective  covariance matrices $\Sigma^{\bs{\al}}$ such that,  for the partial sums~\eqref{N_xi} for the processes $\bs{\xi}=\bs{\xi}^{\bs{\al}},$ $\bs{\al}\in \Lambda_N^{d+1},$ one has   the following approximations as $n,N\to\infty :$
\smallskip

{\rm (i)} if $n=o(N)$ then
\[
n^{-1/2} (\bs{M}  \bigl(
\lfloor nt\rfloor) - (N + n t )   \widehat{\bs{\mu}}^\#
\bigr)
=
\bs{W}^{\bs \al} (t) +o_P(1),
\]
where the error term is uniform in $t\in [0,1]$ 
and $\bs{\al}\in \Lambda_{N} ;$
\smallskip

{\rm (ii)} if $n/N\to \nu\in \R_{>0}$ then
\[
n^{-1/2} \bigl(\bs{M}(\lfloor nt\rfloor) -(N + n t)\widehat{\bs{\mu}}^\# \bigr)
 = \nu^{1/2} t\bs{Z}^{\bs \al}+ \bs{W}^{\bs \al} (t) +o_P(1),
\]
where, for any fixed $0<a<b<\infty,$ the error term is uniform in  $t\in [0,1],$   $\bs{\al}\in \Lambda_{N} $ and $\nu\in (a,b);$
\smallskip

{\rm (iii)} if $N=o(n)$ then
\[
N^{1/2}n^{-1}  \bigl(\bs{M}  (
\lfloor nt\rfloor) -(N + n t )   \widehat{\bs{\mu}}^\# \bigr) =  t\bs{Z}^{\bs\al} +  o_P(1),
\]
where the error term is uniform in  $t\in [0,1] $  and  $\bs{\al}\in \Lambda_{N}  .$
\end{coro}

Restating the above results in terms of the P\'olya urn processes $\bs{\xi}(n)$ is a straightforward task.

We will now comment on the above assertions. The uniformity of the $o_P(1)$-term is understood in the following sense: if $\bs{R}\in\R^d$ is the remainder term in the respective representation then, for any fixed $\varepsilon >0$, the probabilities $\pr (\max_{1\le j\le d}|R_j|>\varepsilon )$ vanish uniformly over the indicated set of parameter values.

Further, if $n=o(N)$ (part~(i)) then the dominating term on the right-hand side of~\eqref{Main} is the one with the Kiefer process. As the first arguments in~$K$ will be very close to the ratios  $N_j/N=\widehat{\mu}_j,$ it means that, for any initial urn composition  vector  $\bs{\al}\in \Lambda^{d+1}_N,$  the trajectory
\[
n^{-1/2} \bigl(\bs{M} (\lfloor nt\rfloor ) - (N+ \lfloor nt\rfloor ) \widehat{\bs{\mu}}^\#\bigr), \quad t\in [0,1],
\]
is approximated by $\bs{W}^{\bs \al} (t):=(K(\widehat{\mu}_1,t), \ldots, K(\widehat{\mu}_d,t))$, which is a $d$-dimensio\-nal Brownian motion with the specified covariance matrix~$\Sigma^{\bs{\al}}.$ This is so because, at that time scale, the proportions $\xi_j (\lfloor nt\rfloor )/(N+\lfloor nt\rfloor)=(1+o(1)) \al_j  /N$ of   balls of different colors~$j$ in the urn  vary very little    when $t\in [0,1].$ Therefore the dynamics of $\bs{\xi} (n)$ (and hence that of $\bs{M} (n)$) is close to that of a multivariate random walk with i.i.d.\ jumps.

When $n\asymp N$ (part~(ii)) the first two terms on the right-hand side of~\eqref{Main} are of the same magnitude. This is a transitional regime. When $n\gg N$ (part~(iii)) the  Brownian bridge term is the main one. In this case, the (almost Gaussian)  randomness of the limiting Dirichlet-distributed vector~$\bs{V}$ (resulting in a random trend with an almost Gaussian coefficient for the dynamics of~$\bs{M} (\lfloor nt\rfloor)$) dominates the random zero-mean ``Brownian oscillations" as the time values are very large.

Our second main result provides an upper bound for the  convergence rate  in the central limit theorem for the  vector $\bs{M}(n)$ of the marked agents' location after~$n$ steps. For $k\ge 1,$ denote by $\mathscr{C}^k $ the class of all Borel convex subsets of~$\R^k$. For a non-negative definite symmetric matrix $\Sigma\in\R^{k\times k},$ we denote by $\Phi_\Sigma$   the zero-mean Gaussian distribution on $\R^k$ with covariance matrix~$\Sigma.$

\begin{theo}\label{T2}
Let $d\in \N_{>0}$ and $\delta \in (0,1/(d+1))$ be  fixed numbers.
For any $N>d, $ $n\ge 1 $ and $\bs{\al}\in\Lambda_N^{d+1}$ such that $\al_j\ge \delta N,$ $j=1,\ldots, d+1,$    one has
\[
\sup_{A\in \mathscr{C}^{d}}
 \biggl|\pr \biggl(\frac{\bs{M} (n) - (N+n)\widehat{\bs{\mu}}^\# }{((N+n) n/N)^{1/2} }\in A\biggr)  -\Phi_{ \Sigma^{\bs\al}} (A)\biggr|
 \le c(n^{-1/2}+N^{-1/2} ).
\]
\end{theo}

Here and in what follows, by the letter $c$ (possibly with a subscript) we denote constants that may depend on $\delta$ and $d$ only and may be different even within one and the same formula.

Observe that the  scaling used in the statement of Theorem~\ref{T2} is universal: it works in all the cases  (i)--(iii) in Corollary~\ref{C1}

\begin{rema}\label{Rem_1}
	\rm
Note that one can   state the assertion of Theorem~\ref{T2} in an  equivalent from in terms of the P\'olya urn composition vectors as well. For $\bs{\xi} = \bs{\xi}^{\bs{\al}},$ $\bs{\mu} = \bs{\mu}^{\bs{\al}},$ set
\begin{align}
\label{xi-forma_0}
\bs{\Xi}(n):=  n^{-1/2}(\bs{\xi}^\#(n)- (N+ n)\bs{\mu}^\#), \quad n\ge 1,
\end{align}
and, for $\bs{x}\in\R^k,$ introduce matrices
\begin{align}
\label{Sigma_x}
\Sigma_{\bs x}:=\mbox{diag} (\bs{x}) - \bs{x}^\top \bs{x}\in  \R^{k\times k} ,
\end{align}
where $\mbox{diag}\, (\bs{x})$ is the diagonal matrix with   diagonal entries $x_j,$ $j=1,\ldots, k,$ and~$^\top$ denotes transposition. Then, under the assumptions of our Theorem~\ref{T2}, one equivalently has
\[
\sup_{A\in \mathscr{C}^{d}}\bigl|\pr (\bs{\Xi} (n)\in A) |-\Phi_{(1+n/N)\Sigma_{\bs\mu^{\scriptscriptstyle\#}}} (A)\bigr|
\le c(n^{-1/2}+N^{-1/2} ).
\]
It will actually be more convenient for us to prove this latter bound.
\end{rema}

\section{Proofs}

\begin{proof}[Proof of Theorem~\ref{T1}] Consider a P\'olya urn process $\bs{\xi}=\bs{\xi}^{\bs{\al}},$ $\bs{\al}\in \Lambda_N^{d+1},$   $N>d\ge 1.$ It is well known that the indicator random vectors $\bs{\eta}(n):= \bs{\xi}(n)- \bs{\xi}(n-1),$ $n \ge 1,$ are exchangeable. By the main theorem in~\cite{BlMa73}, given the random vector~$\bs{V}=\bs{V}^{\bs{\al}}$ from~\eqref{Conv_to_dir}, the   vectors $\bs{\eta}(n) $ are conditionally i.i.d.,
	\begin{align}
	\label{PrV}
	\pr (\bs{\eta}(1)= \bs{w}(j)|\bs{V})= V_j, \quad j=1,\ldots, d+1,
	\end{align}
where $\bs{w}(j)=(\de_{j1}, \de_{j2},\ldots, \de_{j,d+1})$ are the respective unit coordinate vectors in~$\R^{d+1}$ (here $\de_{jk}$ is the Kronecker delta). Therefore, without loss of generality, one can assume that the sequence $\{\bs{\eta}(n):n\ge 1\}$ is given as follows.  For  $\bs{v}\in \Delta^d$ and  $u\in [0,1],$ introduce the vector-valued function  $\bs{g} (\bs{v},u)= (g_1 (\bs{v},u), \ldots, g_{d+1} (\bs{v},u))$ as
\begin{align}
\label{g}
\bs{g} (\bs{v},u):=\sum_{j=1}^{d+1}\ind (\widehat{v}_{j-1}< u\le \widehat{v}_{j })\bs{w}(j).
	\end{align}
Assuming that  $\{U_n:n\ge 1\}$
is a sequence of i.i.d.\ $(0,1)$-uniform random variables, independent of a given random vector $\bs{V}\sim {\rm Dir}_{d+1} (\bs{\al})$, we set
\[
\bs{\eta}(n):= \bs{g} (\bs{V}, U_n), \quad n\ge 1.
\]
This sequence clearly satisfies~\eqref{PrV} and is conditionally i.i.d.\ given the value of~$\bs{V}.$

Now, recalling that  $\bs{\mu}  = \exn \bs{V}= \bs{\al}/N,$ so that $\widehat \mu_j=N_j/N,$ one has
\begin{align}
\label{Nj}
M_j (n)   &  = N_j + \sum_{k=1}^n (\widehat\xi_j (n) - \widehat{\xi}_j(n-1))
\notag
\\
&
= N_j+ n \sum_{i=1}^j V_i + \sum_{i=1}^j \sum_{k=1}^n (g_i (\bs{V}, U_k) - V_i)
\notag
\\
&
= N_j + n \widehat{\mu}_j +   n \sum_{i=1}^j (V_i- \mu_i)
+  \sum_{k=1}^n \sum_{i=1}^j (g_i (\bs{V}, U_k) - V_i)
\notag
\\
&
=  (N  + n) \widehat{\mu}_j    + n  (\widehat{V}_j -\widehat{\mu}_j)
+  \sum_{k=1}^n  \bigl( \ind (U_k\le  \widehat{V}_j) - \widehat{V}_j\bigr).
\end{align}

The second term in the last line of~\eqref{Nj} can be approximated using the following lemma.

\begin{lemo}
	\label{L1}
	Let $ \{\bs{X}^{\bs{\al}}: \bs{\al}\in \Lambda  \}$ be a family of random vectors such that  $\bs{X}^{\bs{\al}}\sim  {\rm Dir}_{k } (\bs{\al })$ for $\bs{\al}\in  \Lambda_m^k,$ $m\ge k\ge 2$. One can construct random vectors from  that family   on a common probability space together with a sequence of Brownian bridges $\{ W^{0,m}: m\ge 1\}$ such that
	\begin{align}
	\label{StrAppr}
	\widehat{X}_j^{\bs{\al}}= \widehat{\al}_j /m  + m^{-1/2}  W^{0,m} (\widehat{\al}_j  /m) + R   ( m, \widehat{\al}_j  ), \quad j=1,\ldots, k,
	\quad \bs{\al}\in \Lambda_m^k,
	\end{align}
	where
	\begin{align}
	\label{StrApprError}
	\limsup_{m\to\infty}  \frac{m }{\ln m}
	\max_{2\le k \le m}\max_{\bs{\al}\in \Lambda^k_m}
	\max_{1\le j\le k} |R (m, \widehat{\al}_j )|\le C\quad  a.s.
	\end{align}
	for some absolute constant $C<\infty$.
\end{lemo}

Clearly, from~\eqref{StrAppr} one also has the following representation for the original Dirichlet-distributed random vectors: for $j=1,\ldots, k,$
\[
{X}_j^{\bs{\al}}=  {\al}_j /m  + m^{-1/2} (  W^{0,m} (\widehat{\al}_j  /m)- W^{0,m} (\widehat{\al}_{j-1}  /m)) + R _j ( m, \widehat{\al}_j  )-R _j ( m, \widehat{\al}_{j-1}  ).
\]

\begin{proof}[Proof of Lemma~\ref{L1}]
Let $\{U'_l: l\ge 1\}$ be a sequence of i.i.d.\ $(0,1)$-uniform random variables. Fix   $m\ge k\ge  2$ for the moment and denote by $U'_{1:m-1}\le U'_{2:m-1}\le\cdots \le U'_{m-1:m-1}$ the order statistics for the sample $\{U'_l: 1\le  l< m \}.$ Set   $U'_{0:m-1}:=0,$  $U'_{m:m-1}:=1.$ Due to the well-known fact that, for any $\bs{\al}\in \Lambda^k_m,$ one has (see, e.g.,~\cite{BaRa98})
\begin{align}
\label{order_Dir}
(U'_{ \widehat \al_1:m-1}-U'_{ \widehat \al_0 :m-1},  U'_{ \widehat \al_2:m-1}- U'_{\widehat \al_1:m-1 },\ldots, U'_{\widehat\al_k:m-1}- U'_{\widehat\al_{k-1}:m-1})   \sim  {\rm Dir}_{k } (\bs{\al }),
\end{align}
one can assume without loss of generality that
$\widehat{\bs{X}}{}^{\bs{\al}} = (U'_{\widehat\al_1:m-1},\ldots,  U'_{\widehat\al_k:m-1}) $ for all $\bs{\al}\in \Lambda^k_m$ (recall also that $\widehat \al_k=m$ for such~$\bs{\al}$).

Next recall that the quantile process for our uniform sample $\{U'_l: 1\le  l< m \}$ is defined~\cite{CsRe78} as
\[
q_{m-1}(u):= (m-1)^{1/2} \biggl( \sum_{i=1}^{m-1} U'_{i:m-1}\ind \Bigl( \frac{i-1}{m-1} <u\le\frac{i }{m-1}  \Bigr)-u\biggr),
\quad u\in [0,1].
\]
It is easily seen  that 
\begin{align}
\label{Xq}
m^{1/2}(\widehat{X}_j^{\bs{\al}}- \widehat{\al}_j /m)
 & =
 (1-m^{-1})^{-1/2}  q_{m-1}(\widehat{\al}_j/ m) , \quad j=1, \ldots, k-1
\end{align}
(note that for $j=k$ the value of the  quantity on the left-hand side of the above formula  is~$0$). Now the desired assertion~\eqref{StrAppr}, \eqref{StrApprError}   immediately  follows from the   strong approximation theorem  for uniform quantile processes $q_{m-1}$ (see, e.g., Theorem~B and  Remark~1 in~\cite{CsRe78}) and the observation that the effect of the factor $(1-m^{-1})^{-1}$ in~\eqref{Xq} is negligible. The latter fact is a consequence of the standard exponential bound  for   the Brownian bridge process:
\begin{align}
\label{BC_bound}
\pr \Bigl(\max_{y\in [0,1]}|W^{0}(y) |>x \Bigr) \le 2 e^{-2 x^2},\quad x>0
\end{align}
(see, e.g., Section~9 in~\cite{Bi99}). Indeed, combined with the Borel--Cantelli lemma, this bound implies that, for any sequence $\{W^{0,m}: m\ge 1\}$ of Brownian bridges on a common probability space, one has
\begin{align}
\label{BC_BB}
\sum_{m=1}^\infty \ind \Bigl(\max_{y\in [0,1]}|W^{0,m}(y) |> \ln ^{1/2} m\Bigr) <\infty\quad \rm a.s.
\end{align}
This, in turn, implies that the additional additive approximation error caused by the above-mentioned factor in~\eqref{Xq}  is a.s.\ $O(m^{-1}\ln ^{1/2} m).$ Lemma~\ref{L1} is proved.
\end{proof}

\begin{rema}
\rm In the case of a fixed $k$ and a sequence  $\bs{\al}(m)\in \Lambda^k_m,$ $m\ge k,$  satisfying the condition   $\max_{1\le j\le k-1}m^{-1/2} |\widehat \al_j (m) - r_j m|\to 0$ for   fixed $0<r_1< \cdots< r_{k-1}<1,$ the asymptotic normality of $ \widehat{\bs{X}}{}^{\bs{\al}(m)}$ as $m\to \infty$    follows from Theorem~2 in~\cite{Wa68}.  The case of growing $k=k(m)\le m$     was considered under the condition that $k(m)/\min_{1\le j\le k(m)}\al_j(m)\to 0 $ as $m\to\infty$  in~\cite{IkMa72}. It was proved in Theorem~3.1 of that paper  that, for a sequence  $\bs{\al}(m)\in \Lambda^{k(m)}_m,$ $m\ge 2,$ the total variation distance between the distribution of $\widehat{\bs{X}}{}^{\bs{\al}(m)}$ and the normal distribution in~$\R^{k(m)}$ with the matching  mean vector and covariance matrix tends to zero as $m\to\infty.$ Upper bounds for the rate of this convergence were obtained in~\cite{Ma75}. The strong approximation result stated in our Lemma~\ref{L1} shows that the asymptotic normality holds uniformly and without any additional assumptions on~$\bs{\al}.$
\end{rema}

Return to the proof of Theorem~\ref{T1}. As $\bs{V}\sim \mbox{Dir}_{d+1}(\bs{\al}),$ by Lemma~\ref{L1} (where we choose $m=N,$  $k=d+1 $), for a suitably constructed family $\{\bs{V}= \bs{V}^{\bs{\al}}: \bs{\al}\in \bs{\Lambda}\}$ and sequence of Brownian bridges $\{W^{0,N}:N\ge 1\},$ as $N\to\infty$ one has
	\begin{align}
	\label{V-mu}
 \widehat{V}_j -\widehat{\mu}_j = N^{-1/2} W^{0,N} (\widehat{\mu}_j)
  + O (N^{-1 }\ln N)\quad \rm a.s.,
	\end{align}
where the $O$-term is understood in the sense of~\eqref{StrApprError}.

Further, denoting by $F_n^*$ the empirical distribution function for the sample $\{U_l: 1\le l \le n\},$  we see that
the last term on the right-hand side of~\eqref{Nj} is  equal to $ \phi_n (\widehat{V}_j),$ where
\[
\phi_n (y) :=n ( F_n^* (y )- y),\quad y\in [0,1].
\]
By the Koml\'os--Major--Tusn\'ady theorem~\cite{KoMaTu75} (see also Theorem~A and Remark~1 in~\cite{CsRe78}), one can construct the sequence $\{ U_n\}$ on a common probability space with a Kiefer process $K$ such that
	\begin{align}
	\label{Z_n}
\phi_n (y) = K(y,n) + R^*(y,n), \quad y\in [0,1], \ n\ge 1,
	\end{align}
where for the remainder term $R^*$ one has
\[
\limsup_{n\to \infty}\sup_{y\in [0,1]} \frac{|R^*(y,n)|}{\ln^2 n} < C^* \quad \rm  a.s.
\]
for some absolute constant $C^*<\infty.$ Combining  now~\eqref{Nj},  \eqref{V-mu} and~\eqref{Z_n} (where we substitute the expression for~$\widehat{V}_j $ from $\eqref{V-mu}$) completes the proof of Theorem~\ref{T1}.
\end{proof}

\begin{proof}[Proof of Corollary~\ref{C1}]   It follows from~\eqref{Main} that, for $k=1,\ldots, n,$ one has
\begin{multline}
n^{-1/2} (M_j (k) - (N+ k)\widehat{\mu}_j)
   =
 \frac{k}{n}\Bigl(\frac{n}N\Bigr)^{1/2} W^{0,N} (\widehat\mu_j)
 +R_j kn^{-1/2}N^{-1}\ln N
 \\
   +
 n^{-1/2}K \bigl(\widehat{\mu}_j +\theta_{\widehat{\mu}_j,N}  , k \bigr)
   + R_j^* n^{-1/2}\ln^2 k,
\label{C1_rep}
\end{multline}
where, in view of~\eqref{BC_BB} and~\eqref{R1_error}, with probability~1,   one has
\begin{align}
\label{theta}
\max_{1\le d <N}
\max_{\bs{\al}\in \Lambda_N^{d+1}}
 \max_{1\le j\le  d}|\theta_{\widehat{\mu}_j,N}|\le 2 N^{-1/2}\ln^{1/2}N \quad \mbox{ for all sufficiently large~$N.$}
\end{align}

(i)  If $n=o(N)$ then the first term on the right-hand side of~\eqref{C1_rep} is clearly $ o_P(1)$ in view of~\eqref{BC_bound}.   The second and forth terms are vanishing a.s.\ in view of the bounds~\eqref{R1_error} and~\eqref{R1*_error}.

For the third term on the right-hand side of~\eqref{C1_rep}, setting $h_n:= n^{-1/3}$ and observing that $ N^{-1/2}\ln^{1/2}N =o(h_n),$ we see from~\eqref{theta} that a.s.\  for all sufficiently large $N$ one  has, for all $k\ge 1,$
\begin{align}
\label{Kmod}
|K(\widehat{\mu}_j + \theta_{\widehat{\mu}_j,N}, k)
 - K (\widehat{\mu}_j,k ) |
 \le
 \sup_{0\le u<v\le u+h_n \le 1}|K(u, k)
 - K (v,k ) |.
\end{align}
Set $\beta_n:=(2n h_n |\ln h_n|)^{-1/2}= n^{-1/3}\bigl(\frac23 \ln n\bigr)^{-1/2}.$ Since  clearly $|\ln h_n| \gg \ln\ln n$ as $n\to\infty,$ by Theorem~1.15.2 in~\cite{CsRe81} we have
\[
\lim_{n\to \infty} \beta_n\sup_{0\le u<v\le u+h_n \le 1}|K(u, n)
- K (v,n ) |=1 \quad \rm a.s.
\]
Hence the maximum of the right-hand   side of~\eqref{Kmod} over $k=1,\ldots,n$ will be a.s.\ bounded by $n^{1/3}$ for all sufficiently large~$n$. Therefore replacing in~\eqref{C1_rep} the term $ n^{-1/2}K  (\widehat{\mu}_j +\theta_{\widehat{\mu}_j,N}, k  )$  with $ n^{-1/2}K  (\widehat{\mu}_j   , k  )$ will introduce an error that will uniformly be $o(1)$ a.s.

It remains to notice that
\begin{align}
\label{Self_sim}
\{  n^{-1/2}K(y, n t): y\in [0,1], t\ge 0\}\deq \{K(y, t): y\in [0,1],
 t\ge 0\}, \quad n\ge 1 ,
\end{align}
(which is obvious from~\eqref{Kiefer_cov}) and that
\[
\{(K(\widehat\mu_1, t),\ldots , K(\widehat\mu_d, t)):   t\ge 0\} \deq \{\bs{W}^{\bs{\al}} (t) : t\ge 0\}
\]
since  both processes are continuous zero-mean Gaussian with a common covariance structure. That the claimed approximation holds for points $t$ that are not multiples of $1/n$ and is uniform follows from the continuity of the approximating processes.

\smallskip

(ii)~Here we assume  that $n/N\to \nu\in \R_{>0}.$ The only difference in the proof from part~(i) is what happens to the first term on the right-hand side of~\eqref{C1_rep}. For $k= \lfloor nt \rfloor,$ it is now equal to  $(1+o(1))t\nu^{1/2} W^{0,N} (\widehat\mu_j)$, where clearly  $\bigl(W^{0,N} (\widehat\mu_1), \ldots, W^{0,N} (\widehat\mu_d)\bigr)\deq \bs{Z}^{\bs{\al}}.$ Finally, we have to recall that the processes $W^{0,N}$ and $K$ were independent of each other, which implies that $\bs{Z}^{\bs{\al}}$ and $\bs{W}^{\bs \al}$ in our approximation are also independent.

\smallskip

(iii) Applying the scaling used in the case when $N=o(n),$ \eqref{C1_rep} turns into
\begin{align*}
N^{ 1/2}n^{-1 }  (N_j (k) - (N+ k)\widehat{\mu}_j)
& =
\frac{k}n W^{0,N} (\widehat\mu_j)
+R_j \frac{k}n N^{-1/2}\ln N
\notag
\\
& +
\Bigl(\frac{N}n\Bigr)^{1/2}n^{-1/2}K \bigl(\widehat{\mu}_j +\theta_N  , k \bigr)
+ R_j^*\Bigl(\frac{N}n\Bigr)^{1/2} \frac{\ln^2 k}{n^{1/2}}.
\end{align*}
For $k=\lfloor nt\rfloor$, the first term on the right hand side is $(1+o(1) )t  W^{0,N} (\widehat\mu_j),$ yielding the approximating term $t\bs{Z}^{\bs{\al}}$.  The second and forth terms are vanishing a.s.\ due to the bounds~\eqref{R1_error} and~\eqref{R1*_error}. Finally, it follows from~\eqref{Self_sim} and~\eqref{BC_bound} that the third   term on the right hand side is~$o_P(1)$. The corollary is proved.
\end{proof}

\begin{proof}[Proof of Theorem~\ref{T2}]
Using notation~\eqref{xi-forma_0}, we will start with re-writing  representation~\eqref{Nj}  as
\begin{align}
\label{xi-form}
\bs{\Xi}(n) 
  = n^{1/2} (\bs{V}^\#   - \bs{\mu}^\#)  +\bs{Y} (\bs{V},n),
\end{align}
where, recalling notation~\eqref{g}, we set
\[
\bs{Y}(\bs{v},n):= n^{-1/2}\sum_{k=1}^n \bs{\gamma}(\bs{v},  k) ,
\quad
\bs{\gamma}(\bs{v},  k) := \bs{g}^\# (\bs{v}, U_k) -   \bs{v}^\#,
\quad
\bs{v}\in \Delta^{d}, \quad k\ge 1.
\]
As one could expect from Corollary~\ref{C1}, it will turn out  that the two terms on the right-hand side of~\eqref{xi-form} are asymptotically independent and normal, with a common   correlation matrix.

It follows from~\eqref{xi-form} that, for any $A\in \mathscr{C}^d  $,
\begin{align}
\label{p_int}
\pr \bigl(\bs{\Xi}(n)\in A\bigr)
 = \int_{\Delta^d}
  \pr \bigl(  \bs{Y} (\bs{v},n) \in A
   -  n^{1/2} (\bs{v}^\#   - \bs{\mu}^\#) \bigr)\pr (\bs{V}   \in d\bs{v} ).
\end{align}
The rest of the proof will consist of the following steps. First we will show that, up to a uniform additive error term $O(n^{-1/2}+N^{-1/2}),$  the integrand here equals
\[
\Phi_{\Sigma_{\bs \mu^{\scriptscriptstyle\#}}} ( A -  n^{1/2} (\bs{v}^\#   - \bs{\mu}^\#))
\]
(recall that $\Sigma_{\bs x} $  was defined in~\eqref{Sigma_x}). With the integrand replaced by this expression, the integral  on the right-hand side of~\eqref{p_int} is just the value on the set~$A$ of the convolution of $\Phi_{\Sigma_{\bs \mu^{\scriptscriptstyle\#}}}$ with the distribution of~$\bs{V}^\#-\bs{\mu}^\#$. To complete the argument, we will apply a known result on the  convergence rate in total variation of the joint distribution of sample quantiles (which coincides with the distribution of $\widehat{\bs{V}}^\#$) to the respective  Gaussian law.

It is easily seen that, for any $\bs{v}\in \Delta^d,$ the random vectors $\bs{\gamma} (\bs{v} ,k),$ $k\ge 1,$ are i.i.d.,
$
\exn \bs{\gamma} (\bs{v} ,1)=\bs{0},$ $\Co (\bs{\gamma}  (\bs{v} ,1))=\Sigma_{\bs  v^{\scriptscriptstyle\#}}. $ Therefore,  from  the Berry--Esseen type multivariate bound (see, e.g., relation~(1) in~\cite{Sa71}) we conclude that, for some constant~$c<\infty$ that depends on~$d$ only, one has
\begin{align}
 \label{BE_bound}
 \sup_{A_0\in\mathscr{C}^{d}  }
 \bigl|\pr \bigl(  \bs{Y} (\bs{v},n) \in A_0 \bigr)
 & - \Phi_{\Sigma_{\bs{v}^{\scriptscriptstyle  \#}}} (A_0 )\bigr|
   \le c n^{-1/2}\exn|\bs{\gamma} (\bs{v} ,1)
 \Sigma^{-1}_{\bs{v}^{\scriptscriptstyle  \#}}
 \bs{\gamma}^\top (\bs{v} ,1)|^{3/2}
 \notag \\
  & \le c n^{-1/2}\|\Sigma^{-1}_{\bs{v}^{\scriptscriptstyle  \#}}\| \exn\|\bs{\gamma} (\bs{v} ,1)\|^3
  \le  c n^{-1/2}\|\Sigma^{-1}_{\bs{v}^{\scriptscriptstyle  \#}}\|
\end{align}
since clearly $\|\bs{\gamma} (\bs{v} ,1)\| \le 2.$ One can easily verify by a direct computation that, setting $\bs{l}:=(1,\ldots, 1)\in \R^d,$ one has
\begin{align}
\label{Sigma-1}
\Sigma^{-1}_{\bs{u} }= \mbox{diag}(\bs{u}^{-1}) +\frac{\bs{l}^\top \bs{l}}{1-\widehat{\bs{u}}_d},\quad \bs{u}\in \Delta^{d\#}.
\end{align}
So for the matrix operator norm in~\eqref{BE_bound} we get
\begin{align}
\label{MOp_norm}
\|\Sigma^{-1}_{\bs{v}^{\scriptscriptstyle  \#}}\|
 \le
 \max_{1\le j\le d } v_j^{-1} + dv_{d+1}^{-1}
 <
  d\sum_{j=1}^{d+1} v_j^{-1}.
\end{align}

Now we will bound the difference between $\Phi_{\Sigma_{\bs{v}^{\scriptscriptstyle  \#}}}  (A_0 )$ and $\Phi_{\Sigma_{\bs{\mu}^{\scriptscriptstyle  \#}}}  (A_0 )$. For random vectors $\bs{Y}'\sim P', \bs{Y}''\sim P''$ taking values in $\R^k,$ $k\ge 1,$ denote  by
\begin{align}
\label{d_TV}
d_{TV}(\bs{Y}',\bs{Y}'')
= d_{TV}(P',P'')
:  =
\sup_B |P'( B)-P''( B)|
 \equiv  \sup_g |\exn g (\bs{Y}' ) - \exn g (\bs{Y}'')|
\end{align}
the total variation distance between the distributions of these   vectors (the first supremum is taken over all Borel~$B\subset \R^k$, the second one over all measurable functions $g:\R^k\to [0,1];$ see, e.g., Chapter~3 in~\cite{RaKlStFi13}). For a matrix $\Ac= (a_{ij}: 1\le i,j\le k)\in\R^{k\times k},$ denote by $\|\Ac\|_F:=\bigl( \sum_{i,j=1}^k a_{ij}^2\bigr)^{1/2}$ its Frobenius norm. By Theorem~1.1 in~\cite{DeMeRe18},
\begin{align}
\label{d_TV_Phi}
d_{TV} (\Phi_{\Sigma_{\bs{v}^{\scriptscriptstyle \#}} }, \Phi_{\Sigma_{\bs{\mu}^{\scriptscriptstyle \#}} } )
 \le
 2 \|\Sigma_{\bs{\mu}^{\scriptscriptstyle \#}} ^{-1}\Sigma_{\bs{v}^{\scriptscriptstyle \#}} -\Ic_d\|_F
  \le
    2 \|\Sigma_{\bs{\mu}^{\scriptscriptstyle \#}} ^{-1}\|_F\cdot \|\Sigma_{\bs{v}^{\scriptscriptstyle \#}} - \Sigma_{\bs{\mu}^{\scriptscriptstyle \#}} \|_F
\end{align}
as $\|\cdot \|_F$ is sub-multiplicative. From~\eqref{Sigma-1} one clearly has
\[
\|\Sigma^{-1}_{\bs{\mu}^{\scriptscriptstyle \#}} \|_F^2
\le
2 \sum_{j=1}^d \mu_j^{-2} + 2 d^2 \mu_{d+1}^{-2}
 <
2d(d+1)\delta^{-2}
\]
under the assumptions of Theorem~\ref{T2}.

The last factor in~\eqref{d_TV_Phi} does not exceed
\begin{align*}
\|\mbox{diag} (\bs{v}^{\scriptscriptstyle \#})
 - \mbox{diag} (\bs{\mu}^{\scriptscriptstyle \#}) \|_F
 +
 \|  (\bs{v}^{\scriptscriptstyle \#})^\top \bs{v}^{\scriptscriptstyle \#}
 - (\bs{\mu}^{\scriptscriptstyle \#})^\top \bs{\mu}^{\scriptscriptstyle \#} \|_F,
\end{align*}
where the square of first term is $\sum_{j=1}^d (v_j-\mu_j)^2 $, whereas that of the second one is
\begin{align*}
  \sum_{i,j=1}^d (v_i v_j-\mu_i \mu_j)^2
   \le
   2\sum_{i,j=1}^d  v_i^2( v_j- \mu_j)^2
   +  2\sum_{i,j=1}^d  \mu_j^2( v_i- \mu_i)^2
    \le 4 \sum_{j=1}^d (v_j-\mu_j)^2
\end{align*}
Thus we have got from~\eqref{d_TV_Phi} the bound
\begin{align*}
d_{TV} (\Phi_{\Sigma_{\bs{v}^{\scriptscriptstyle \#}} }, \Phi_{\Sigma_{\bs{\mu}^{\scriptscriptstyle \#}} } )
 \le c\|\bs{v}-\bs{\mu}\|.
\end{align*}

Combining this with \eqref{BE_bound} and~\eqref{MOp_norm}  yields
\begin{align}
\label{BE_bound1}
\sup_{A_0\in \mathscr{C}^{d}  }
\bigl|\pr \bigl(  \bs{Y} (\bs{v},n) \in A_0 \bigr) - \Phi_{\Sigma_{\bs{\mu}^{\scriptscriptstyle \#}} }(A_0 )\bigr|
\le c R_n  (\bs{v},\bs{\mu}),
\end{align}
where
\[
 R_n  (\bs{v},\bs{\mu}):=
  n^{-1/2} \sum_{j=1}^{d+1}  v_j^{-1 } +  \|\bs{v} -\bs{\mu} \|.
\]
To use this bound in~\eqref{p_int}, we need to compute the expectation of~$ R_n  (\bs{V},\bs{\mu}).$ First note that since  $\bs{V}\sim  \mbox{Dir}_{d+1}(\bs{\al})$, the   components  $V_j$ of this vector are  beta-distributed with respective parameters $( \al_j, N - \al_j),$ $j=1,\ldots, d+1$. Hence
\begin{align*}
\exn  V_j^{-1 }
  = \frac{\Be( \al_j-1, N - \al_j )}{\Be(\al_j, N - \al_j)}
  =  \frac{\Gamma(\alpha_j-1)\Gamma(N )}{\Gamma(N -1)\Gamma(\alpha_j )}
  =\frac{N-1}{\alpha_j-1}< \frac2{\delta}
\end{align*}
as we can assume without loss of generality that $\alpha_j\ge 2.$ As the variance of the beta distribution with parameters $a, b>0$ equals $ab (a+b)^{-2}(a+b+1)^{-1}$, we get
\begin{align*}
\bigl(\exn \|\bs{V} -\bs{\mu} \|\bigr)^2
& \le   \exn \|\bs{V} -\bs{\mu}   \|^2
=
  \sum_{j=1}^{d+1}\exn \bigl(  V_j  -\mu_j \bigr)^2
 \notag
 \\
 &
  =
  \sum_{j=1}^{d+1}\Va (  V_j)
  =
  \frac1{N+1}\sum_{j=1}^{d+1}\mu_j (1-\mu_j)
  \le  \frac{d+1}{4(N+1)}.
\end{align*}
We conclude that
\begin{align}
\label{ER_4}
\exn R_n (\bs{V},\bs{\mu}) \le c(n^{-1/2}+ N^{-1/2}) .
\end{align}
Now we obtain  from~\eqref{p_int},  \eqref{BE_bound1} and~\eqref{ER_4}  that, for $A\in\mathscr{C}^d ,$
\begin{align}
\label{upper_b}
\bigl|\pr  (\bs{\Xi}(n)\in A )
-
 \exn 
\Phi_{\Sigma_{\bs{\mu}^{\scriptscriptstyle \#}} }(A
-  n^{1/2} (\bs{v}^\#   - \bs{\mu}^\#))   \bigr|
\le  c(n^{-1/2}+ N^{-1/2}).
\end{align}

It follows from~\eqref{order_Dir} (with $k=d+1,$ $m=N$) and Theorem~3.1 in~\cite{Ma75} that, under the assumptions of our Theorem~\ref{T2}, for a random vector  $\bs{\zeta}\sim \Phi_{\Sigma_{\bs{\mu}^{\scriptscriptstyle \#}}},$ one has
\[
d_{TV} (\bs{V}^\#-\bs{\mu}^\#, N^{-1/2}\bs{\zeta}) \le c N^{-1/2}.
\]
Now from~\eqref{upper_b} and the last  relation in~\eqref{d_TV} we see that
\begin{align*}
\label{upper_b}
\bigl|
\pr  (\bs{\Xi}(n)\in A )
-
\exn  \Phi_{\Sigma_{\bs{\mu}^{\scriptscriptstyle \#}} }
\bigl(A -   n^{1/2}  N^{-1/2}   \bs{\zeta}  \bigr) \bigr|
 \le c(n^{-1/2}+ N^{-1/2}).
\end{align*}
As clearly $\exn \Phi_{\Sigma_{\bs{\mu}^{\scriptscriptstyle \#}} }
\bigl(A -   n^{1/2}  N^{-1/2}   \bs{\zeta}  \bigr) = \Phi_{(1 +n/N)\Sigma_{\bs{\mu}^{\scriptscriptstyle \#}}}
\bigl(A), $  Theorem~\ref{T2} is proved in view of Remark~\ref{Rem_1}.
\end{proof}
	



\begin{thebibliography}{99}
 


\bibitem{AtKa68}
{\sc   Athreya, K.B., and Ney, P.E.}  (1968).
Embedding of urn schemes into continuous
time Markov branching processes and related limit theorems.
{\em Ann.\ Math.\ Statist.} 39, 1801--1817.

\bibitem{AtNe72}
{\sc  Athreya, K.B., and Ney, P.E.} (1972).
{\em Branching Processes.} Springer, New York.


\bibitem{BaRa98}
{\sc  Balakrishnan, N., and Rao, C.R.}   (1998).
Order statistics: An introduction.
In: N.~Balakrishnan, C.R.Rao (eds), {\em Handbook of Statistics.} V.~16. Elsevier Science, Amsterdam, 3--24.

\bibitem{Bi99}
{\sc Billingsley, P.} (1999).
{\em Convergence of Probability Measures.} 2nd edn.
Wiley, New York.

\bibitem{BiLaSi08}
{\sc Binder, B.J., Landman, K.A., and Simpson, M.J.}  (2008)
Modeling proliferative tissue growth: A general approach and an avian case study.
{\em Phis.\ Review E}, 78,   031912.

\bibitem{BiLa09}
{\sc Binder, B.J., and  Landman, K.A.}  (2009).
Exclusion processes on a growing domain.
{\em J.~Theoret.\ Biol.} 259, 541--551.

\bibitem{BlMa73}
{\sc  Blackwell, D., and MacQueen, J.B.}  (1973).
Ferguson distributions via P\'olya urn schemes.
{\em Ann.\ Stat.} 1:2,   353--355.

\bibitem{BlHo91}
{\sc  Blom, G., and Holst, L.}  (1991).
Embedding procedures for discrete problems in probability.
{\em  Math.\ Scientist}, 16,  29--40.


\bibitem{CsRe78}
{\sc Cs\H{o}rg\"o, M., and R\'ev\'esz, P.} (1978).
Strong approximations of the quantile process. {\em Ann.\ Statist.} 6:4,   882--894.


\bibitem{CsRe81}
{\sc Cs\H{o}rg\"o, M., and R\'ev\'esz, P.} (1981).
{\em Strong Approximations in Probability and Statistics. }
Academic Press, New York.

\bibitem{de06}
{\sc de La Fortelle, A.}  (2006).
Yule process sample path asymptotics.
{\em Electron.\ Commun.\ Probab.} 11,  193--199.

\bibitem{DeMeRe18}
{\sc Devroye, L., Mehrabian, A., and Reddad, T.} (2018).
The total variation distance between high-dimensional Gaussians.
https://arxiv.org/abs/1810.08693.

\bibitem{EgPo23}
{\sc  Eggenberger, F., and P\'olya, G.}  (1923).
\"Uber die Statistik verketteter Vorgange.
{\em Z.\ Angew.\ Math.\ Mech.} 3, 279--289.

\bibitem{GoRe13}
{\sc  Goldstein, L., and Reinert, G.}  (2013).
Stein's method for the Beta distribution and the
P\'olya--Eggenberger Urn.
{\em J.~App.\ Prob.} 50:4, 1187-1205.


\bibitem{HeSa55}
{\sc Hewitt, E. and Savage, L. J.} (1955).
Symmetric measures on Cartesian products.
{\em Trans.\ Amer.\ Math.\ Soc.} 80, 470--501.



\bibitem{IkMa72}
{\sc  Ikeda, S., and Matsunawa, T.}  (1972).
On the uniform asymptotic normality of sample quantiles. 	
{\em Ann.\ Inst.\ Stat.\ Math.} 24, 33--52.


\bibitem{Ja04}
{\sc Janson, S.}  (2004).
Functional limit theorems for multitype branching processes and generalized P\'olya urns.
{\em Stoch.\  Proc.\ Appl.} 110:2, 177--245.

\bibitem{Ja06}
{\sc  Janson, S.} (2006).
Limit theorems for triangular urn schemes.
{\em Probab.\ Theory Related Fields,} 134:3,  417--452.

\bibitem{JoKo77}
{\sc Johnson, N.L., and Kotz, S.}  (1977).
{\em Urn Models and Their Applications.}
New York,   Springer.

 

\bibitem{KoMaTu75}
{\sc  Komlos, J., Major, P., and Tusnady, G.}  (1975). An approximation of partial sums of independent RVs and the sample DF. I.
{\em Z.~Wahrscheinlichkeitstheor.\ verw.\ Geb.} 32:1--2,  111--131.

\bibitem{KoBa97}
{\sc Kotz S., and Balakrishnan N.} (1997).
Advances in urn models during the past two decades. In: Balakrishnan N. (ed.), {\em  Advances in Combinatorial Methods and Applications to Probability and Statistics.} Birkh\"auser, Boston, 203--257.

\bibitem{Ma09}
{\sc Mahmoud, H.M.}  (2009).
{\em P\'olya Urn Models.} CRC Press, Boca Raton.

\bibitem{Ma06}
{\sc Markov, A.A.} (1906).
Extension of the law of large numbers to quantities depending on each other.
{\em Izv. fizm.-mat.\ obsch. Kazanskom univ.} 2:15, 135--156. (In Russian.) [Reprinted in:
{\em J.~Electron.\ Hist.\ Probab.\ Stat.}  2:1b (2006),
 Article 10,   http://eudml.org/doc/128778.]

\bibitem{Ma75}
{\sc Matsunawa, T.}  (1975).
On the error evaluation of the joint normal approximation for sample quantiles.
{\em Ann.\ Inst.\ Statist.\ Math.} 27:2, 189--199.

\bibitem{RaKlStFi13}
{\sc Rachev, S.T., Klebanov, L.B.,Stoyanov,  S.V., and Fabozzi, F.} (2013).
{\em The Methods of Distances in the Theory of Probability and Statistics.}
Springer, New York.


\bibitem{Sa71}
{\sc Sazonov,  V.V.} (1975).
On a bound for the rate of convergence in the multidimensional central limit theorem.
In: {\em   Proc.\ Sixth Berkeley Symp.\ on Math.\ Stat.\ and Prob.,} Berkeley and Los Angeles, University of California Press. Vol.~2,   563--582.

 



\bibitem{Wa68}
{\sc Walker, A.M.}  (1968).
A note on the asymptotic distribution of sample quantiles.
{\em J.~R.\ Stat.\ Soc.\ Ser.~B,} 30:3, 570--575.

	\end{thebibliography}
\end{document}